\documentclass[10pt,reqno]{amsart}
\usepackage[T2A]{fontenc}
\usepackage{amsmath}
\usepackage{amssymb}
\usepackage{amsfonts}

\usepackage{mathtools} 
\usepackage{trimclip} 
\usepackage{color}
\usepackage{longtable}
\usepackage{caption}

\newtheorem{thm}{Theorem}
\newtheorem{prop}[thm]{Proposition}
\newtheorem{lem}[thm]{Lemma}
\newtheorem{cor}[thm]{Corollary}

\newcommand{\veqt}{\mathord{%
\marginbox{.45ex 0ex 0ex 0ex}{\clipbox{.35ex 0ex 0ex -.03ex}{$q$}}%
\kern -.825ex\clipbox{.45ex .2ex 0ex -.03ex}{$\color{white}q$}%
\kern -.654ex\clipbox{.85ex .2ex 0ex -.03ex}{$\color{white}q$}%
\kern -.154ex\clipbox{.85ex .2ex 0ex -.03ex}{$\color{white}q$}%
\kern -.75ex\varepsilon\kern .1ex}%
}

\newcommand{\veqs}{\mathord{%
\marginbox{.4ex 0ex 0ex 0ex}{\clipbox{.35ex 0ex 0ex -.03ex}{$\scriptstyle q$}}%
\kern -.585ex\clipbox{.45ex .2ex 0ex -.03ex}{$\color{white}{\scriptstyle q}$}%
\kern -.664ex\clipbox{.6ex .2ex 0ex -.03ex}{$\color{white}{\scriptstyle q}$}%
\kern -.354ex\clipbox{.45ex .2ex 0ex -.03ex}{$\color{white}{\scriptstyle q}$}%
\kern -.55ex\varepsilon\kern .1ex}%
}

\newcommand{\veq}{\mathchoice{\veqt}{\veqt}{\veqs}{}}

\begin{document}
\renewcommand{\refname}{References}
\renewcommand{\proofname}{Proof\,}
\thispagestyle{empty}

\title[On the maximal tori]{On the maximal tori in finite\\ linear and unitary groups}
\author{{Andrei V. Zavarnitsine}}%
\address{Andrei V. Zavarnitsine
\newline\hphantom{iii} Sobolev Institute of Mathematics,
\newline\hphantom{iii} 4, Koptyug av.
\newline\hphantom{iii} 630090, Novosibirsk, Russia}%
\email{zav@math.nsc.ru}%


\maketitle {\small
\begin{quote}
\noindent{\sc Abstract.} To follow up on the results of \cite{07ButGr.t}, we propose a computationally efficient explicit
cyclic decomposition of the maximal tori in the groups $\operatorname{SL}_n(q)$ and $\operatorname{SU}_n(q)$ and
their projective images. We also derive some corollaries to simplify practical calculation of the maximal tori.
The result is based on a generic cyclic decomposition of a finite abelian group which might also be of interest.


\end{quote}
}

\section{Introduction}

The maximal tori in finite groups of Lie type have been extensively studied. For
$\operatorname{SL}_n(q)$ and $\operatorname{SU}_n(q)$ as well as their projective versions $\operatorname{PSL}_n(q)$
and $\operatorname{PSU}_n(q)$, the structure of maximal tori is clarified in \cite[Theorems 2.1, 2.2]{07ButGr.t}.
The conjugacy classes of maximal tori in these groups are parameterized by unordered partitions of $n$.
The cyclic decomposition form \cite{07ButGr.t} for a maximal torus corresponding to the
partition $n=n_1+\ldots +n_s$ is canonical, but it has combinatorial computational growth as the
number $s$ of components of the partition grows, see Theorem \ref{bg} below. We propose another cyclic
decomposition for these tori which might be useful in practical computations, see Theorem \ref{main}.
It is based on a generic cyclic decomposition of a finite abelian group stated in Proposition \ref{pdec}.

In order to formulate the main result, we introduce some notation. For a nonzero integer $n$, we denote by
$\mathbb{Z}_n$
a cyclic group of order $|n|$. Let $q$ be a prime power. Denote $\operatorname{(P)SL}_n(-q)=\operatorname{(P)SU}_n(q)$
and let $\varepsilon=\pm 1$.
Throughout, the ligature $\veq$ stands for the product $\varepsilon q$. The $\operatorname{gcd}$ and $\operatorname{lcm}$ of nonzero
integers $n_1,\ldots, n_s$ are assumed to be positive and denoted by
$(n_1,\ldots,n_s)$ and $[n_1,\ldots,n_s]$, respectively.

\begin{thm}\label{main} Let $T$ be a maximal torus of $\operatorname{SL}_n(\veq)$ parameterized by the
partition $n=n_1+\ldots +n_s$. Denote
\begin{align}
\begin{split}\label{aidef}
  a_1 & =\veq^{(n_1,\ldots,n_s)}-1, \\
  a_2 & =[\veq^{n_1}-1,\veq^{(n_2,\ldots,n_s)}-1], \\
  a_3 & =[\veq^{n_2}-1,\veq^{(n_3,\ldots,n_s)}-1], \\
   & \vdots \\
  a_{s-1} & =[\veq^{n_{s-2}}-1,\veq^{(n_{s-1},n_s)}-1], \\
  a_s & =[\veq^{n_{s-1}}-1,\veq^{n_s}-1].
\end{split}
\end{align}
Then
\begin{equation}\label{tdec}
  T\cong \mathbb{Z}_{a_1'}\times \mathbb{Z}_{a_2}\times\ldots\times \mathbb{Z}_{a_s},
\end{equation}
where $a_1'=a_1/(\veq-1)$.

Let $\overline{T}$ be the image of $T$ in $\operatorname{PSL}_n(\veq)$. Set
$$
d=(n,\veq-1),\qquad d'=\left(n/(n_1,\ldots,n_s),\veq-1\right).
$$
Relabelling $a_2,\ldots,a_s$ arbitrarily if necessary, denote
\begin{align}
\begin{split}\label{bidef}
  b_2 & =(a_2,\ldots,a_s), \\
  b_3 & =[a_2,(a_3,\ldots,a_s)], \\
   & \vdots \\
  b_{s-1} & =[a_{s-2},(a_{s-1},a_s)], \\
  b_s & =[a_{s-1},a_s].
\end{split}
\end{align}
Then
\begin{equation}\label{ovtdec}
\overline{T}\cong \mathbb{Z}_{b_1'}\times \mathbb{Z}_{b_2'}\times\mathbb{Z}_{b_3}\ldots\times \mathbb{Z}_{b_s},
\end{equation}
where $b_1'=d'a_1/d(\veq-1)$ and $b_2'=b_2/d'$.
\end{thm}

Theorem \ref{main} allows us to give a simplified cyclic decomposition of maximal tori of
$\operatorname{SL}_n(\veq)$ in many particular partition cases. For example, the following rather general fact holds.

\begin{cor}\label{cmain}
Let $T$ be a maximal torus of $\operatorname{SL}_n(\veq)$ parameterized by the
partition $n=n_1+\ldots +n_s$. Denote
$t=(n_1,\ldots,n_s)$ and let $n_i=tn_i'$ for $i=1,\ldots,s$.
If $(n_i',n_j')=1$ for $i\ne j$ then we have
\begin{equation}\label{ninj}
T\cong \mathbb{Z}_{(\veq^t-1)/(\veq-1)}\times \mathbb{Z}_{(\veq^{n_i}-1)(\veq^{n_j}-1)/(\veq^t-1)}
\times \prod_{\substack{\mathclap{k=1,\ldots,s}\\ k\ne i,j}}\mathbb{Z}_{\veq^{n_k}-1}.
\end{equation}
In particular,
\renewcommand{\labelitemi}{$\bullet$}
\begin{itemize}
\item if $n_i'=1$ for some $i$ then
\begin{equation}\label{cni1}
T\cong \mathbb{Z}_{(\veq^t-1)/(\veq-1)}\times \ \prod_{\substack{\mathclap{k=1,\ldots,s}\\ k\ne i}}\ \mathbb{Z}_{\veq^{n_k}-1};
\end{equation}
\item if $(n_i,n_j)=1$ for $i\ne j$ then
$$
T\cong \mathbb{Z}_{(\veq^{n_i}-1)(\veq^{n_j}-1)/(\veq-1)} \times \prod_{\substack{\mathclap{k=1,\ldots,s} \\ k\ne i,j}}\mathbb{Z}_{\veq^{n_k}-1};
$$
\item if $n_1=\ldots=n_s\ \ (=t)$ then
\begin{equation}\label{enieq}
T\cong \mathbb{Z}_{(\veq^t-1)/(\veq-1)}\times \mathbb{Z}_{\veq^t-1}^{\,s-1}.
\end{equation}
\end{itemize}
\end{cor}

Decomposition (\ref{enieq}) can also be readily deduced from \cite{07ButGr.t}, see Theorem \ref{bg} below.
\medskip

\noindent{\bf Example 1.} Let the decomposition be  $n=1+2+3+4+5+6$.  Applying directly the result of \cite{07ButGr.t}
gives the following structure of the corresponding torus of $\operatorname{SL}_{21}(\veq)$:
$$
T\cong \mathbb{Z}_{\veq-1}^2\times \mathbb{Z}_{\veq^2-1}\times \mathbb{Z}_{(\veq^3-1)(\veq+1)}
\times \mathbb{Z}_{(\veq^6-1)(\veq^4+\veq^3+\veq^2+\veq+1)(\veq^2+1)},
$$
whereas (\ref{cni1}) yields
$$
T\cong \mathbb{Z}_{\veq^2-1}\times \mathbb{Z}_{\veq^3-1}
\times \mathbb{Z}_{\veq^4-1}\times \mathbb{Z}_{\veq^5-1}\times \mathbb{Z}_{\veq^6-1}.
$$

\noindent{\bf Example 2.} Let $\varepsilon = -1$ and let $n=3+6+6+9$. Then $t=3$ and (\ref{cni1}) implies that the
corresponding torus of $\operatorname{SU}_{24}(q)$ has the structure
$$
T\cong \mathbb{Z}_{q^2-q+1}\times \mathbb{Z}_{q^6-1}\times \mathbb{Z}_{q^6-1}\times \mathbb{Z}_{q^9+1}.
$$

\medskip

Expression (\ref{ninj}) alone allows us to explicitly write down by hand decompositions of all
maximal tori of $\operatorname{SL}_n(\veq)$ for $n\leqslant 30$. For example, the tori of $\operatorname{SL}_{10}(\veq)$ are
listed in Table \ref{sut}. The following case, however, is not covered by (\ref{ninj}).

\medskip \noindent{\bf Example 3.} Let $n=6+10+15$. Then, depending on the
ordering of $n_i$'s, the corresponding torus of $\operatorname{SL}_{31}(\veq)$ can be decomposed by Theorem \ref{main} in three ways
\begin{align*}
T \  \cong & \ \ \mathbb{Z}_{(\veq^{15}-1)(\veq+1)}\times \mathbb{Z}_{(\veq^{10}-1)(\veq^4+\veq^2+1)} \\
   \cong & \ \ \mathbb{Z}_{(\veq^{10}-1)(\veq^2+\veq+1)}\times \mathbb{Z}_{(\veq^{15}-1)(\veq^3+1)} \\
   \cong & \ \ \mathbb{Z}_{(\veq^6-1)(\veq^4+\veq^3+\veq^2+\veq+1)}\times \mathbb{Z}_{(\veq^{15}-1)(\veq^5+1)},
\end{align*}
whereas \cite{07ButGr.t} gives a fourth decomposition
$$
T \  \cong \ \mathbb{Z}_{(\veq^{15}-1)(\veq^5+1)(\veq^2-\veq+1)}\times \mathbb{Z}_{(\veq^5-1)(\veq^2+\veq+1)(\veq+1)}.\\
$$

\medskip
Nevertheless, we can generalize (\ref{ninj}) to include this case as follows.

\begin{cor}\label{cmaing} In the notation of Corollary \ref{cmain}, if $(n_i',n_j',n'_k)=1$ for pairwise distinct $i,j,k$ then
$$
T\cong \mathbb{Z}_{(\veq^t-1)/(\veq-1)}\times \mathbb{Z}_{[\veq^{n_i}-1,\veq^{(n_j,n_k)}-1]}
\times \mathbb{Z}_{[\veq^{n_j}-1,\veq^{n_k}-1]}
\times \prod_{\substack{\mathclap{l=1,\ldots,s}\\ l\ne i,j,k}}\mathbb{Z}_{\veq^{n_l}-1}.
$$
\end{cor}
A similar generalization can be inferred from Theorem \ref{main} for arbitrarily many coprime numbers $n_i'$.

The projective case is somewhat more complicated. The following particular partitions of $n$ yield a simplified decomposition
of $\overline{T}$.

\begin{cor}\label{corp}
Let $\overline{T}$ be the image in $\operatorname{PSL_n}(\veq)$ of a maximal torus of $\operatorname{SL}_n(\veq)$ parameterized by the
partition $n=n_1+\ldots +n_s$. Denote $t=(n_1,\ldots,n_s)$, $d=(n,\veq-1)$, and $d'=(n/t,\veq-1)$.

\noindent
$(i)$ If $s=1$ then $\overline{T}\cong \mathbb{Z}_{(\veq^n-1)/d(\veq-1)}$.

\noindent
$(ii)$ If $s=2$ then
\begin{equation}\label{s2}
\overline{T}\cong \mathbb{Z}_{d'(\veq^t-1)/d(\veq-1)} \times \mathbb{Z}_{(\veq^{n_1}-1)(\veq^{n_2}-1)/d'(\veq^t-1)}.
\end{equation}

\noindent
$(iii)$ If $n_i=n_j=1$ for $i\ne j$ then
$$
\overline{T}\cong \mathbb{Z}_{(\veq-1)/d}\times\  \prod_{\substack{\mathclap{r=1,\ldots,s}\\ r\ne i,j}}\ \mathbb{Z}_{\veq^{n_r}-1};
$$

\noindent
$(iv)$ If $n_i=1$ and $(n_j,n_k)=1$ for pairwise distinct $i,j,k$, then
$$
\overline{T}\cong \mathbb{Z}_{(\veq-1)/d}
\times \mathbb{Z}_{(\veq^{n_j}-1)(\veq^{n_k}-1)/(\veq-1)}
\times \ \ \prod_{\substack{\mathclap{r=1,\ldots,s}\\ \mathclap{r\ne i,j,k}}}\ \ \mathbb{Z}_{\veq^{n_r}-1}.
$$

\noindent
$(v)$ If $(n_i,n_j)=1$ and $(n_k,n_l)=1$ for pairwise distinct $i,j,k,l$, then
$$
\overline{T}\cong \mathbb{Z}_{(\veq-1)/d}
\times \mathbb{Z}_{(\veq^{n_i}-1)(\veq^{n_j}-1)/(\veq-1)}
\times \mathbb{Z}_{(\veq^{n_k}-1)(\veq^{n_l}-1)/(\veq-1)}
\times \ \ \prod_{\substack{\mathclap{r=1,\ldots,s}\\ \mathclap{r\ne i,j,k,l}}}\ \ \mathbb{Z}_{\veq^{n_r}-1}.
$$

\noindent
$(vi)$ Assume that $n_i=t$ for some $i$. Set
$r=\operatorname{gcd}\{n_l\mid l\ne i\}$.

\begin{itemize}
\item[$(vi.1)$] If $n_j=r$ for $j\ne i$ then
$$
\overline{T}\cong \mathbb{Z}_{d'(\veq^t-1)/d(\veq-1)} \times \mathbb{Z}_{(\veq^r-1)/d'}
\times \prod_{\substack{\mathclap{l=1,\ldots,s}\\ l\ne i,j}}\mathbb{Z}_{\veq^{n_l}-1}.
$$

\noindent
In particular, if $n_1=\ldots=n_s\ \ (=t)$ then
\begin{equation}\label{pnieq}
\overline{T}\cong  \mathbb{Z}_{d'(\veq^t-1)/d(\veq-1)} \times \mathbb{Z}_{(\veq^t-1)/d'}\times \mathbb{Z}_{\veq^t-1}^{s-2}.
\end{equation}

\item[$(vi.2)$] If $(n_j,n_k)=r$ for pairwise distinct $i,j,k$ then
\begin{equation}\label{iv2}
\overline{T}\cong \mathbb{Z}_{d'(\veq^t-1)/d(\veq-1)} \times \mathbb{Z}_{(\veq^r-1)/d'}\times \mathbb{Z}_{(\veq^{n_j}-1)(\veq^{n_k}-1)/(\veq^r-1)}
\times \prod_{\substack{\mathclap{l=1,\ldots,s}\\ l\ne i,j,k}}\mathbb{Z}_{\veq^{n_l}-1}.
\end{equation}
\end{itemize}
\end{cor}

Observe that decompositions (\ref{s2}) and (\ref{pnieq}) can also be readily deduced from \cite{07ButGr.t}.

\medskip
\noindent{\bf Example 4.} Let $n=3+6+9+12$. Then $t=3$ and $r=3$. We may set $n_j=6$, $n_k=9$.
By (\ref{iv2}), the image of $T$ in $\operatorname{PSL}_{30}(\veq)$ is
$$
\overline{T}\cong \mathbb{Z}_{(\veq^2+\veq+1)/d_3}\times \mathbb{Z}_{(\veq^3-1)/d_{10}} \times \mathbb{Z}_{(\veq^9-1)(\veq^3+1)}
\times \mathbb{Z}_{\veq^{12}-1},
$$
where $d_3=(3,\veq-1)$ and $d_{10}=(10,\veq-1)$.

\medskip
In Table \ref{sut}, we give decompositions for all images $\overline{T}$ in $\operatorname{PSL_{10}}(\veq)$.
Most of them are consequences of Corollary \ref{corp}.

\section{A cyclic decomposition of finite abelian groups}

Let $m_1,\ldots,m_s\in\mathbb{N}$.
The direct product of cyclic groups
$$
A=\mathbb{Z}_{m_1}\times\ldots\times\mathbb{Z}_{m_s}
$$
has a cyclic direct factor of order $d_1=(m_1,\ldots,m_s)$. In other words, $A \cong \mathbb{Z}_{d_1}\times A'$
for an abelian group $A'$. We are interested in an explicit cyclic decomposition of~$A'$. One such decomposition can be
obtained canonically. We have
\begin{equation}\label{can}
A=\mathbb{Z}_{d_1}\times\ldots\times\mathbb{Z}_{d_s},
\end{equation}
where $d_k=\delta_k/\delta_{k-1}$, $k=1,\ldots,s$, and
\begin{equation}\label{dk}
\delta_k=\operatorname{gcd}\{m_{i_1}\cdot\ldots\cdot m_{i_k}\mid 1\leqslant i_1<\ldots<i_k\leqslant s\},
\end{equation}
$k=0,\ldots,s$, is the {\em $k$-th determinant divisor} of the matrix $\operatorname{diag}(m_1,\ldots,m_s)$, i.\,e.
the $\operatorname{gcd}$ of all its $k\times k$ minors. Clearly, (\ref{can}) provides a decomposition for $A'$:
$$
A'=\mathbb{Z}_{d_2}\times\ldots\times\mathbb{Z}_{d_s}.
$$
There are two alternative descriptions of the invariants $d_k$'s. The first one uses prime factorization.
Given\footnote{The notation $p^\mu \parallel m$ stands for the fact that the prime power $p^\mu$ {\em exactly divides}
$m$, i.\,e. $p^\mu\mid m$ and $p^{\mu+1}\nmid m$.}
a prime $p$, let
$p^{\mu_{p,i}}\parallel m_i$ for suitable $\mu_{p,i}\geqslant 0$, $i=1,\ldots,s$. Also, let
$\nu_{p,1}\leqslant\ldots\leqslant\nu_{p,s}$ be such that
$$
\{\{\,\nu_{p,1},\ldots,\nu_{p,s}\,\}\}=\{\{\,\mu_{p,1},\ldots,\mu_{p,s}\,\}\}
$$
is the equality of multisets (i.\,e. sets with repetitions). Then
\begin{equation}\label{dkp}
\delta_k=\prod_p p^{\nu_{p,1}+\ldots+\nu_{p,k}}, \qquad  d_k=\prod_p p^{\nu_{p,k}},
\end{equation}
$k=1,\ldots,s$, the products being taken over all primes. This readily follows from (\ref{dk}) and the fact that
$$
\operatorname{min}\{ \mu_{p,i_1}+\ldots +\mu_{p,i_k} \mid 1\leqslant i_1<\ldots<i_k\leqslant s\}= \nu_{p,1}+\ldots+\nu_{p,k}
$$
for every $p$. The second description is
\begin{equation}\label{dk2}
  d_k = \operatorname{lcm}\{ (m_{i_1},\ldots, m_{i_{s-k+1}})\mid 1\leqslant i_1<\ldots<i_{s-k+1}\leqslant s\},
\end{equation}
$k=1,\ldots,s$, which follows from (\ref{dkp}) and the fact that
$$
\operatorname{max}\{\operatorname{min}\{   \mu_{p,i_1},\ldots, \mu_{p,i_{s-k+1}}\}\mid 1\leqslant i_1<\ldots<i_{s-k+1}\leqslant s \}=\nu_{p,k}.
$$
for every $p$.

The canonical decomposition (\ref{can}) is explicit but is not computationally efficient because of
the combinatorial growth of the number of arguments on the right-hand side of both (\ref{dk}) and (\ref{dk2}), or
due to the dependence on prime factorization in (\ref{dkp}).

The proof of Theorem \ref{main} is based on the following fact which yields
an alternative cyclic decomposition for $A'$.

\begin{prop}\label{pdec} For $m_1,\ldots,m_s\in \mathbb{N}$, we have
\begin{equation}\label{mdec}
\mathbb{Z}_{m_1}\times\ldots\times\mathbb{Z}_{m_s}\cong \ \mathbb{Z}_{a_1}\times\ldots\times\mathbb{Z}_{a_s},
\end{equation}
where
\begin{align}
  a_1 & =(m_1,\ldots,m_s),  \nonumber \\
  a_2 & =[m_1,(m_2,\ldots,m_s)], \nonumber \\
  a_3 & =[m_2,(m_3,\ldots,m_s)], \nonumber \\
   & \vdots \label{ai} \\
  a_{s-1} & =[m_{s-2},(m_{s-1},m_s)], \nonumber \\
  a_s & =[m_{s-1},m_s]. \nonumber
\end{align}
\end{prop}
\begin{proof}
Fix a prime $p$. Let $p^{\mu_i}\parallel m_i$ and let $p^{\alpha_i}\parallel a_i$, $i=1,\ldots,s$. Set $\mu_0=0$. By~(\ref{ai}), we have
\begin{equation} \label{ali}
  \alpha_i = \operatorname{max}\{\mu_{i-1},  \operatorname{min}\{\mu_i,\ldots,\mu_s\}  \}
\end{equation}
for $i=1,\ldots,s$. It suffices to show that the equality of multisets $\{\{\mu_1,\ldots,\mu_s\}\}=\{\{\alpha_1,\ldots,\alpha_s\}\}$ holds.
Define $i_0,i_1,\ldots$ by the rule $i_0=0$ and $i_k$ is such that
\begin{equation} \label{nui}
  \mu_{i_k} = \operatorname{min}\{\mu_{i_{k-1}+1},\mu_{i_{k-1}+2},\ldots,\mu_s\}
\end{equation}
for $k=1,\ldots,l$, where $l$ is the smallest index with $\mu_{i_l}=\mu_s$. An explicit bijection between
$\mu_1,\ldots,\mu_s$ and $\alpha_1,\ldots,\alpha_s$ is then given as follows.
For every $k=0,\ldots, l-1$, we have
by (\ref{ali}) and (\ref{nui})
$$
\alpha_{i_k+1}=\operatorname{max}\{\mu_{i_k},\operatorname{min}\{\mu_{i_k+1},\ldots,\mu_s\}\}=
\operatorname{max}\{\mu_{i_k},\mu_{i_{k+1}}\}=\mu_{i_{k+1}}
$$
and
$$
\alpha_i = \operatorname{max}\{\mu_{i-1},  \operatorname{min}\{\mu_i,\ldots,\mu_s\}\}=
\operatorname{max}\{\mu_{i-1}, \mu_{i_{k+1}}\}=\mu_{i-1},
$$
where $i_k+2\leqslant i\leqslant i_{k+1}$. The claim follows.
\end{proof}

Observe that we might as well have proven Proposition \ref{pdec} by induction on $s$ without using prime factorization.
Also, we emphasize that in general the decomposition on the right-hand side of (\ref{mdec}) is not canonical and may essentially depend on
the ordering of $m_i$'s.

\section{Auxiliary  facts}

We state explicitly the necessary facts form \cite{07ButGr.t} slightly modifying
the original notation.

\begin{thm}[\mbox{\cite[Theorems 2.1, 2.2]{07ButGr.t}}]\label{bg}
Let $n \geqslant 2$ and let $T$ be a maximal torus of $\operatorname{SL}_n(\veq)$ parameterized by the
partition $n=n_1+\ldots +n_s$. For $k=1,\ldots,s$, denote
\begin{equation}\label{dkbg}
d_k=\operatorname{lcm}\{(\veq^{n_{i_1}}-1,\ldots,\veq^{n_{i_{s-k+1}}}-1)\mid 1\leqslant i_1<\ldots<i_{s-k+1}\leqslant s \}.
\end{equation}
Then $d_k$ divides $d_{k'}$ for $k\leqslant k'$ and
$$
T\cong \mathbb{Z}_{d_1'}\times \mathbb{Z}_{d_2}\times\ldots\times \mathbb{Z}_{d_s},
$$
where $d_1'=d_1/(\veq-1)$. Let $\overline{T}$ be the image of $T$ in $\operatorname{PSL}_n(\veq)$.
Denote
$$
d=(n,\veq-1),\qquad d'=\left(n/(n_1,\ldots,n_s),\veq-1\right).
$$
If $s=1$ then $\overline{T}\cong \mathbb{Z}_{d_1'}$, where $d_1'=d_1/d(\veq-1)$. If $s>1$ then
$$
\overline{T}\cong \mathbb{Z}_{d_1'}\times \mathbb{Z}_{d_2'}\times \mathbb{Z}_{d_3}\times\ldots \times\mathbb{Z}_{d_s},
$$
where $d_1'=d'd_1/d(\veq-1)$ and $d_2'=d_2/d'$.
\end{thm}

The following number-theoretic result will also be used.

\begin{lem}\label{eqab} Let $a,b,q\in \mathbb{N}$, let $\varepsilon=\pm 1$, and let $\veq=\varepsilon q$.
Then up to sign we have
$$
(\veq^a-1,\veq^b-1)=\veq^{(a,b)}-1.
$$
\end{lem}
\begin{proof} We use \cite[Lemma 6(iii)]{04Zav}. If either $\varepsilon=1$ or both $a,b$ even, we have
$$
(q^a-1,q^b-1)=q^{(a,b)}-1.
$$
If $\varepsilon=-1$, $a$ even, $b$ odd, we have
$$
(q^a-1,q^b+1)=q^{(a,b)}+1.
$$
If $\varepsilon=-1$, $a$ odd, $b$ even, we have
$$
(q^a+1,q^b-1)=q^{(a,b)}+1.
$$
If $\varepsilon=-1$, $a$ odd, $b$ odd, we have
$$
(q^a+1,q^b+1)=q^{(a,b)}+1.
$$
The claim follows.
\end{proof}

\section{Proof of main results}

We first prove Theorem \ref{main}.
\begin{proof} We may assume $n\geqslant 2$. Denote $m_i=\veq^{n_i}-1$, $i=1,\ldots, s$, and $d_1=(m_1,\ldots,m_s)$. Let
$$
A=\mathbb{Z}_{m_1}\times \ldots \times \mathbb{Z}_{m_s}.
$$
and $A_1=\mathbb{Z}_{d_1}$. By (\ref{can}), we have $A\cong A_1\times A'$, where
$$
A'=\mathbb{Z}_{d_2}\times \ldots \times \mathbb{Z}_{d_s}
$$
and the $d_k$'s are given by equations (\ref{dk2}) which are clearly the same as~(\ref{dkbg}). In particular,
Theorem \ref{bg} implies $T\cong A_1/A_1'\times A'$, where $A_1'\cong \mathbb{Z}_{\veq-1}$ is a cyclic subgroup of $A_1$.
On the other hand, Proposition \ref{pdec} yields
\begin{equation}\label{app}
A'\cong \mathbb{Z}_{a_2}\times \ldots \times \mathbb{Z}_{a_s}
\end{equation}
and the $a_k$'s are given by (\ref{ai}). Lemma \ref{eqab} implies that the $a_k$'s are the same as defined in (\ref{aidef}).
Hence the required decomposition (\ref{tdec}) holds.

We now consider the image $\overline{T}$ of $T$ in $\operatorname{PSL}_n(\veq)$. We may assume $s>1$.
The argument is similar except that we now consider $A'$ in place of $A$. We have
$$
A\cong A_1\times A_2\times A'',
$$
where $A_2=\mathbb{Z}_{d_2}$. Theorem \ref{bg} implies $\overline{T}\cong A/(A_1''\times A_2'')$, where
$A_1''\cong \mathbb{Z}_{d(\veq-1)/d'}$ and $A_2''\cong \mathbb{Z}_{d'}$  are cyclic subgroups of $A_1$ and $A_2$, respectively.
Since $A_2\times A''\cong A'$, decomposition (\ref{app}) implies $d_2=(a_2,\ldots,a_s)$. Moreover, Proposition~\ref{pdec}
with $a_2,\ldots,a_s$ playing the role of $m_1,\ldots,m_s$ and ordered arbitrarily yields
$$
A''\cong \mathbb{Z}_{b_3}\times \ldots \times \mathbb{Z}_{b_s},
$$
where the $b_i$'s are the same as defined in (\ref{bidef}). Therefore, we have the required decomposition (\ref{ovtdec}) for $\overline{T}$.
\end{proof}

We now prove Corollary \ref{cmain}.
\begin{proof} First, let us make no assumptions on the components $n_i$'s. In the notation of Theorem \ref{main}, define
\begin{equation}\label{ani}
A_{n_1,\ldots,n_s}(\veq)=\mathbb{Z}_{a_2}\times\ldots\times \mathbb{Z}_{a_s}.
\end{equation}
Then (\ref{tdec}) implies
\begin{equation}\label{ann}
T\cong \mathbb{Z}_{a_1'}\times A_{n_1,\ldots,n_s}(\veq).
\end{equation}
Setting $t=(n_1,\ldots,n_s)$ we can rewrite (\ref{ann}) as
\begin{equation}\label{tan}
T\cong \mathbb{Z}_{(\veq^t-1)/(\veq-1)}\times A_{n_1',\ldots,n_s'}(\veq^t),
\end{equation}
where $n_i=tn_i'$ for $i=1,\ldots,s$.

Now let $(n_i,n_j)=1$ for distinct $i,j$. Since the ordering of $n_i$'s is not fixed, we may assume that
$\{i,j\}=\{s-1,s\}$ (although we could proceed without this assumption). Then (\ref{aidef}) implies that $a_1=\veq-1$,
$a_2=\veq^{n_1}-1$, \ldots, $a_{s-1}=\veq^{n_{s-2}}-1$, and $a_s=[\veq^{n_{s-1}}-1,\veq^{n_s}]=(\veq^{n_i}-1)(\veq^{n_j}-1)/(\veq-1)$.
Therefore, Theorem \ref{main} and expression (\ref{ani}) imply
\begin{equation}\label{tnij}
A_{n_1,\ldots,n_s}(\veq ) \cong  \mathbb{Z}_{(\veq^{n_i}-1)(\veq^{n_j}-1)/(\veq-1)} \times \prod_{\substack{\mathclap{k=1,\ldots,s} \\ k\ne i,j}}\mathbb{Z}_{\veq^{n_k}-1}.
\end{equation}

Finally, if $(n_i',n_j')=1$ for distinct $i,j$ then both (\ref{tan}) and (\ref{tnij}) yield the required decomposition (\ref{ninj}).
\end{proof}

Corollary \ref{cmaing} can be proved similarly. We also outline a proof of Corollary \ref{corp}.

\begin{proof} Items $(i)$ and $(ii)$ are straightforward. They also readily follow from Theorem~\ref{bg}.
We show $(v)$. The proof of $(iii)$ and $(iv)$ is similar and simpler. As above we may assume that $\{i,j\}=\{s-1,s\}$ to
obtain $a_1=\veq-1$, $a_2=\veq^{n_1}-1$, \ldots, $a_{s-1}=\veq^{n_{s-2}}-1$, and $a_s=[\veq^{n_{s-1}}-1,\veq^{n_s}]=(\veq^{n_i}-1)(\veq^{n_j}-1)/(\veq-1)$.
Now, we may also assume $\{k,l\}=\{s-3,s-2\}$ and relabel the last three $a_i$'s so that
$a_{s-2}=(\veq^{n_i}-1)(\veq^{n_j}-1)/(\veq-1)$, $a_{s-1}=\veq^{n_{s-3}}-1$, $a_s=\veq^{n_{s-2}}-1$. This does not affect
the validity of isomorphism (\ref{ovtdec}), since we did not assume that the ordering of $a_2,\ldots, a_s$ is fixed when proving Theorem \ref{main}.
Thus, (\ref{bidef}) implies $b_2=\veq-1$, $b_3=a_2=\veq^{n_1}-1$, \ldots, $b_{s-2}=a_{s-3}=\veq^{n_{s-4}}-1$, $b_{s-1}=a_{s-2}=(\veq^{n_i}-1)(\veq^{n_j}-1)/(\veq-1)$,
$b_s=[a_{s-1},a_s]=[\veq^{n_{s-3}}-1,\veq^{n_{s-2}}-1]=(\veq^{n_k}-1)(\veq^{n_l}-1)/(\veq-1)$. The claim now follows by (\ref{ovtdec}) because $d=d'$ in this case.

We now prove $(vi.2)$. The proof of $(vi.1)$ is similar.
We may assume that $i=s$. Then $a_1=\veq^t-1$,
$a_l=\veq^{n_{l-1}}-1$, $l=2,\ldots,s$. Also, assume that $\{j,k\}=\{s-2,s-1\}$.
Then $b_2=\veq^r-1$, $b_3=a_2$, \ldots, $b_{s-1}=a_{s-2}$, $b_s=(\veq^{n_i}-1)(\veq^{n_j}-1)/(\veq^r-1)$. The claim now follows by (\ref{ovtdec}).
\end{proof}

\bigskip

\captionsetup{width=.9\textwidth}
\begin{longtable}[htb]{l@{\hspace{7.5pt}}l@{\hspace{7.5pt}}l}
\caption{The maximal tori of $\operatorname{SL}_{10}(\protect\veq)$ and their images in $\operatorname{PSL}_{10}(\protect\veq)$. \\
Notation: $\protect\veq=\pm q$, $d_2=(2,\protect\veq-1)$, $d_5=(5,\protect\veq-1)$, $d=d_2d_5=(10,\protect\veq-1)$.\label{sut}}\\
$[n_1,...,n_s]$ & $\operatorname{SL}_{10}(\veq)$  & $\operatorname{PSL}_{10}(\veq)\vphantom{p_{p_{p_p}}}$  \\
\hline
$[ 1^{10} ]$      &$\mathbb{Z}_{\veq-1}^{\,9}\vphantom{A^{A^A}}$                                      &$\mathbb{Z}_{\veq-1}^{\,8}\times\mathbb{Z}_{(\veq-1)/d}$                                          \\[1.9pt]
$[ 2, 1^8]$       &$\mathbb{Z}_{\veq^2-1}\times\mathbb{Z}_{\veq-1}^{\,7}$                                &$\mathbb{Z}_{\veq^2-1}\times\mathbb{Z}_{\veq-1}^{\,6}\times\mathbb{Z}_{(\veq-1)/d}$                  \\[1.9pt]
$[ 2^2, 1^6 ]$    &$\mathbb{Z}_{\veq^2-1}^{\,2}\times\mathbb{Z}_{\veq-1}^{\,5}$                          &$\mathbb{Z}_{\veq^2-1}^{\,2}\times\mathbb{Z}_{\veq-1}^{\,4}\times\mathbb{Z}_{(\veq-1)/d}$            \\[1.9pt]
$[ 2^3, 1^4 ]$    &$\mathbb{Z}_{\veq^2-1}^{\,3}\times\mathbb{Z}_{\veq-1}^{\,3}$                          &$\mathbb{Z}_{\veq^2-1}^{\,3}\times\mathbb{Z}_{\veq-1}^{\,2}\times\mathbb{Z}_{(\veq-1)/d}$            \\[1.9pt]
$[ 2^4, 1^2  ]$   &$\mathbb{Z}_{\veq^2-1}^{\,4}\times\mathbb{Z}_{\veq-1}$                                &$\mathbb{Z}_{\veq^2-1}^{\,4}\times\mathbb{Z}_{(\veq-1)/d}$                                        \\[1.9pt]
$[ 2^5 ]$         &$\mathbb{Z}_{\veq^2-1}^{\,4}\times\mathbb{Z}_{\veq+1}$                                &$\mathbb{Z}_{\veq^2-1}^{\,3}\times\mathbb{Z}_{(\veq^2-1)/d_5}\times\mathbb{Z}_{(\veq+1)/d_2}$                                 \\[1.9pt]
$[ 3, 1^7 ]$      &$\mathbb{Z}_{\veq^3-1}\times\mathbb{Z}_{\veq-1}^{\,6}$                                &$\mathbb{Z}_{\veq^3-1}\times\mathbb{Z}_{\veq-1}^{\,5}\times\mathbb{Z}_{(\veq-1)/d}$                               \\[1.9pt]
$[ 3, 2, 1^5]$    &$\mathbb{Z}_{\veq^3-1}\times\mathbb{Z}_{\veq^2-1}\times\mathbb{Z}_{\veq-1}^{\,4}$        &$\mathbb{Z}_{\veq^3-1}\times\mathbb{Z}_{\veq^2-1}\times\mathbb{Z}_{\veq-1}^{\,3}\times\mathbb{Z}_{(\veq-1)/d}$      \\[1.9pt]
$[ 3,  2^2, 1^3]$ &$\mathbb{Z}_{\veq^3-1}\times\mathbb{Z}_{\veq^2-1}^{\,2}\times\mathbb{Z}_{\veq-1}^{\,2}$  &$\mathbb{Z}_{\veq^3-1}\times\mathbb{Z}_{\veq^2-1}^{\,2}\times\mathbb{Z}_{\veq-1}\times\mathbb{Z}_{(\veq-1)/d}$      \\[1.9pt]
$[ 3, 2^3, 1]$    &$\mathbb{Z}_{\veq^3-1}\times\mathbb{Z}_{\veq^2-1}^{\,3}$                              &$\mathbb{Z}_{(\veq^3-1)(\veq+1)}\times\mathbb{Z}_{\veq^2-1}^{\,2}\times\mathbb{Z}_{(\veq-1)/d}$                               \\[1.9pt]
$[ 3^2, 1^4 ]$    &$\mathbb{Z}_{\veq^3-1}^{\,2}\times\mathbb{Z}_{\veq-1}^{\,3}$                          &$\mathbb{Z}_{\veq^3-1}^{\,2}\times\mathbb{Z}_{\veq-1}^{\,2}\times\mathbb{Z}_{(\veq-1)/d}$                              \\[1.9pt]
$[ 3^2, 2, 1^2 ]$ &$\mathbb{Z}_{\veq^3-1}^{\,2}\times\mathbb{Z}_{\veq^2-1}\times\mathbb{Z}_{\veq-1}$        &$\mathbb{Z}_{\veq^3-1}^{\,2}\times\mathbb{Z}_{\veq^2-1}\times\mathbb{Z}_{(\veq-1)/d}$            \\[1.9pt]
$[ 3^2, 2^2 ]$    &$\mathbb{Z}_{(\veq^3-1)(\veq+1)}\!\times\mathbb{Z}_{\veq^3-1}\!\times\mathbb{Z}_{\veq^2-1}$     &$\mathbb{Z}_{(\veq^3-1)(\veq+1)}^2\times\mathbb{Z}_{(\veq-1)/d}$     \\[1.9pt]
$[ 3^3, 1 ]$      &$\mathbb{Z}_{\veq^3-1}^{\,3}$                                                      &$\mathbb{Z}_{\veq^3-1}^{\,2}\times\mathbb{Z}_{(\veq^3-1)/d}$                                                       \\[1.9pt]
$[ 4, 1^6 ]$      &$\mathbb{Z}_{\veq^4-1}\times\mathbb{Z}_{\veq-1}^{\,5}$                                &$\mathbb{Z}_{\veq^4-1}\times\mathbb{Z}_{\veq-1}^{\,4}\times\mathbb{Z}_{(\veq-1)/d}$                                 \\[1.9pt]
$[ 4, 2, 1^4]$    &$\mathbb{Z}_{\veq^4-1}\times\mathbb{Z}_{\veq^2-1}\times\mathbb{Z}_{\veq-1}^{\,3}$        &$\mathbb{Z}_{\veq^4-1}\times\mathbb{Z}_{\veq^2-1}\times\mathbb{Z}_{\veq-1}^{\,2}\times\mathbb{Z}_{(\veq-1)/d}$         \\[1.9pt]
$[ 4, 2^2, 1^2]$  &$\mathbb{Z}_{\veq^4-1}\times\mathbb{Z}_{\veq^2-1}^{\,2}\times\mathbb{Z}_{\veq-1}$        &$\mathbb{Z}_{\veq^4-1}\times\mathbb{Z}_{\veq^2-1}^{\,2}\times\mathbb{Z}_{(\veq-1)/d}$         \\[1.9pt]
$[ 4, 2^3]$       &$\mathbb{Z}_{\veq^4-1}\times\mathbb{Z}_{\veq^2-1}^{\,2}\times\mathbb{Z}_{\veq+1}$        &$\mathbb{Z}_{\veq^4-1}\!\times\mathbb{Z}_{\veq^2-1}\!\times\mathbb{Z}_{(\veq^2-1)/d_5}\!\times\mathbb{Z}_{(\veq+1)/d_2}$         \\[1.9pt]
$[ 4, 3, 1^3]$    &$\mathbb{Z}_{\veq^4-1}\times\mathbb{Z}_{\veq^3-1}\times\mathbb{Z}_{\veq-1}^{\,2}$        &$\mathbb{Z}_{\veq^4-1}\times\mathbb{Z}_{\veq^3-1}\times\mathbb{Z}_{\veq-1}\times\mathbb{Z}_{(\veq-1)/d}$         \\[1.9pt]
$[ 4, 3, 2, 1]$   &$\mathbb{Z}_{\veq^4-1}\times\mathbb{Z}_{\veq^3-1}\times\mathbb{Z}_{\veq^2-1}$            &$\mathbb{Z}_{(\veq^3-1)(\veq+1)}\times\mathbb{Z}_{\veq^4-1}\times\mathbb{Z}_{(\veq-1)/d}$             \\[1.9pt]
$[ 4, 3^2 ]$      &$\mathbb{Z}_{(\veq^4-1)(\veq^2+\veq+1)}\times\mathbb{Z}_{\veq^3-1}$                         &$\mathbb{Z}_{(\veq^4-1)(\veq^2+\veq+1)}\times\mathbb{Z}_{(\veq^3-1)/d}$                          \\[1.9pt]
$[ 4^2, 1^2 ]$    &$\mathbb{Z}_{\veq^4-1}^{\,2}\times\mathbb{Z}_{\veq-1}$                                &$\mathbb{Z}_{\veq^4-1}^{\,2}\times\mathbb{Z}_{(\veq-1)/d}$                                 \\[1.9pt]
$[ 4^2, 2 ]$      &$\mathbb{Z}_{\veq^4-1}^{\,2}\times\mathbb{Z}_{\veq+1}$                                &$\mathbb{Z}_{\veq^4-1}\times\mathbb{Z}_{(\veq^4-1)/d_5}\times\mathbb{Z}_{(\veq+1)/d_2}$                                \\[1.9pt]
$[ 5, 1^5 ]$      &$\mathbb{Z}_{\veq^5-1}\times\mathbb{Z}_{\veq-1}^{\,4}$                                &$\mathbb{Z}_{\veq^5-1}\times\mathbb{Z}_{\veq-1}^{\,3}\times\mathbb{Z}_{(\veq-1)/d}$                                 \\[1.9pt]
$[ 5, 2, 1^3]$    &$\mathbb{Z}_{\veq^5-1}\times\mathbb{Z}_{\veq^2-1}\times\mathbb{Z}_{\veq-1}^{\,2}$        &$\mathbb{Z}_{\veq^5-1}\times\mathbb{Z}_{\veq^2-1}\times\mathbb{Z}_{\veq-1}\times\mathbb{Z}_{(\veq-1)/d}$         \\[1.9pt]
$[ 5, 2^2, 1]$    &$\mathbb{Z}_{\veq^5-1}\times\mathbb{Z}_{\veq^2-1}^{\,2}$                              &$\mathbb{Z}_{(\veq^5-1)(\veq+1)}\times\mathbb{Z}_{\veq^2-1}\times\mathbb{Z}_{(\veq-1)/d}$                               \\[1.9pt]
$[ 5, 3, 1^2]$    &$\mathbb{Z}_{\veq^5-1}\times\mathbb{Z}_{\veq^3-1}\times\mathbb{Z}_{\veq-1}$              &$\mathbb{Z}_{\veq^5-1}\times\mathbb{Z}_{\veq^3-1}\times\mathbb{Z}_{(\veq-1)/d}$               \\[1.9pt]
$[ 5, 3, 2]$      &$\mathbb{Z}_{\veq^5-1}\times\mathbb{Z}_{(\veq+1)(\veq^3-1)}$                             &$\mathbb{Z}_{(\veq^5-1)(\veq^2+\veq+1)(\veq+1)}\times\mathbb{Z}_{(\veq-1)/d}$                              \\[1.9pt]
$[ 5, 4, 1]$      &$\mathbb{Z}_{\veq^5-1}\times\mathbb{Z}_{\veq^4-1}$                                    &$\mathbb{Z}_{(\veq^5-1)(\veq^3+\veq^2+\veq+1)}\times\mathbb{Z}_{(\veq-1)/d}$                                     \\[1.9pt]
$[ 5^2 ]$         &$\mathbb{Z}_{\veq^5-1}\times\mathbb{Z}_{\veq^4+\veq^3+\veq^2+\veq+1}$                          &$\mathbb{Z}_{(\veq^5-1)/d_2}\times\mathbb{Z}_{(\veq^4+\veq^3+\veq^2+\veq+1)/d_5}$                           \\[1.9pt]
$[ 6, 1^4]$       &$\mathbb{Z}_{\veq^6-1}\times\mathbb{Z}_{\veq-1}^{\,3}$                                &$\mathbb{Z}_{\veq^6-1}\times\mathbb{Z}_{\veq-1}^{\,2}\times\mathbb{Z}_{(\veq-1)/d}$                                 \\[1.9pt]
$[ 6, 2, 1^2]$    &$\mathbb{Z}_{\veq^6-1}\times\mathbb{Z}_{\veq^2-1}\times\mathbb{Z}_{\veq-1}$              &$\mathbb{Z}_{\veq^6-1}\times\mathbb{Z}_{\veq^2-1}\times\mathbb{Z}_{(\veq-1)/d}$               \\[1.9pt]
$[ 6, 2^2 ]$      &$\mathbb{Z}_{\veq^6-1}\times\mathbb{Z}_{\veq^2-1}\times\mathbb{Z}_{\veq+1}$              &$\mathbb{Z}_{\veq^6-1}\times\mathbb{Z}_{(\veq^2-1)/d_5}\times\mathbb{Z}_{(\veq+1)/d_2}$               \\[1.9pt]
$[ 6, 3, 1]$      &$\mathbb{Z}_{\veq^6-1}\times\mathbb{Z}_{\veq^3-1}$                                    &$\mathbb{Z}_{\veq^6-1}\times\mathbb{Z}_{(\veq^3-1)/d}$                                     \\[1.9pt]
$[ 6, 4]$         &$\mathbb{Z}_{(\veq^6-1)(\veq^2+1)}\times\mathbb{Z}_{\veq+1}$                             &$\mathbb{Z}_{(\veq^6-1)(\veq^2+1)/d_5}\times\mathbb{Z}_{(\veq+1)/d_2}$                              \\[1.9pt]
$[ 7, 1^3]$       &$\mathbb{Z}_{\veq^7-1}\times\mathbb{Z}_{\veq-1}^{\,2}$                                &$\mathbb{Z}_{\veq^7-1}\times\mathbb{Z}_{\veq-1}\times\mathbb{Z}_{(\veq-1)/d}$                                 \\[1.9pt]
$[ 7, 2, 1]$      &$\mathbb{Z}_{\veq^7-1}\times\mathbb{Z}_{\veq^2-1}$                                    &$\mathbb{Z}_{(\veq^7-1)(\veq+1)}\times\mathbb{Z}_{(\veq-1)/d}$                                     \\[1.9pt]
$[ 7, 3]$         &$\mathbb{Z}_{(\veq^7-1)(\veq^2+\veq+1)}$                                                 &$\mathbb{Z}_{(\veq^7-1)(\veq^2+\veq+1)/d}$                                                  \\[1.9pt]
$[ 8, 1^2]$       &$\mathbb{Z}_{\veq^8-1}\times\mathbb{Z}_{\veq-1}$                                      &$\mathbb{Z}_{\veq^8-1}\times\mathbb{Z}_{(\veq-1)/d}$                                       \\[1.9pt]
$[ 8, 2]$         &$\mathbb{Z}_{\veq^8-1}\times\mathbb{Z}_{\veq+1}$                                      &$\mathbb{Z}_{(\veq^8-1)/d_5}\times\mathbb{Z}_{(\veq+1)/d_2}$                                       \\[1.9pt]
$[ 9, 1]$         &$\mathbb{Z}_{\veq^9-1}$                                                            &$\mathbb{Z}_{(\veq^9-1)/d}$                                                             \\[1.9pt]
$[ 10]$           &$\mathbb{Z}_{\veq^9+\veq^8+\ldots+\veq+1}\vphantom{p_{p_{p_p}}}$                           &$\mathbb{Z}_{(\veq^9+\veq^8+\ldots+\veq+1)/d}$                                                  \\[1.9pt]
\hline
\end{longtable}

\bigskip

\end{document}